\documentclass[12pt,reqno]{amsart}

\newtheorem{theorem}{Theorem}[section]
\newtheorem{proposition}[theorem]{Proposition}
\newtheorem{corollary}[theorem]{Corollary} 
\newtheorem{lemma}[theorem]{Lemma} 
\newtheorem{definition}[theorem]{Definition}
\theoremstyle{definition} 

\numberwithin{equation}{section} 
\newcommand\bm{\begin{pmatrix}} 
\renewcommand\em{\end{pmatrix}}
\renewcommand\({\Big(} 
\renewcommand\){\Big)}

\newcommand\er{\eqref} 
\newcommand\be[1]{\begin{equation}\label{#1}}
\newcommand\ee{\end{equation}} 

\newcommand\<{\langle}
\renewcommand\>{\rangle}

\newcommand\W{\bigwedge\limits}

\newcommand\la{\alpha}
\newcommand\lb{\beta}
\renewcommand\lg{\gamma}

\newcommand\Le{\epsilon}
\newcommand\lh{\eta}
\newcommand\li{\iota}
\newcommand\lk{\kappa}
\renewcommand\ll{\lambda}

\newcommand\lF{\Phi}

\newcommand\lQ{\Psi}
\newcommand\lp{\pi}

\newcommand\lr{\rho}

\newcommand\lT{\Theta}
\newcommand\Lt{\tau}

\newcommand\lO{\Omega}
\newcommand\lx{\xi}
\newcommand\lX{\Xi}

\newcommand\dl{\partial} 

\newcommand\el{\ell}
 
\newcommand\os{\sim} 

\newcommand\ui{\cap}
\newcommand\yi{\wedge}

\newcommand\ic{\subset}

\newcommand\qi{\mapsto}
\newcommand\xx{\times}
\newcommand\xt{\otimes}

\newcommand\Al{{\mathbf A}} 
\newcommand\Bl{{\mathbf B}} 
\newcommand\Cl{{\mathbf C}}

\newcommand\Kl{{\mathbf K}}

\newcommand\gL{\mathfrak g}
\newcommand\hL{\mathfrak h}
\newcommand\kL{\mathfrak k}
\newcommand\Ll{\mathfrak l}
\newcommand\mL{\mathfrak m}
\newcommand\pL{\mathfrak p}
\newcommand\tL{\mathfrak t}
 
\newcommand\al{\boldsymbol a} 
\newcommand\bl{\boldsymbol b}
\newcommand\cl{\boldsymbol c}
\newcommand\kl{\boldsymbol k}
\newcommand\ol{{\boldsymbol o}}

\renewcommand\d[1]{{\dot{#1}}}

\renewcommand\t[1]{{\widetilde{#1}}} 
 
\renewcommand\o[1]{{\overline{#1}}}

\newcommand\F[1]{{\sqrt{#1}}} 
\newcommand\f[2]{{\frac{#1}{#2}}}

\begin{document} 

\title[Homogeneous bundles on hermitian symmetric spaces]{Homogeneous holomorphic
hermitian principal bundles over hermitian symmetric spaces} 

\author{Indranil Biswas}
\address{School of Mathematics, Tata Institute of Fundamental Research, 1 Homi Bhabha Road, 
Mumbai 400005, India}
\email{indranil@math.tifr.res.in}

\author{Harald Upmeier}
\address{Fachbereich Mathematik und Informatik, Philipps-Universit\"at Marburg, Lahnberge, 
Hans-Meerwein-Strasse, D-35032 Marburg, Germany}
\email{upmeier@mathematik.uni-marburg.de}

\keywords{Irreducible hermitian symmetric space; principal
bundle; homogeneous complex structure; hermitian structure.}

\subjclass[2010]{32M10, 14M17, 32L05}

\begin{abstract}  
We give a complete characterization of invariant integrable complex
structures on principal bundles defined over hermitian symmetric spaces, using the Jordan
algebraic approach for the curvature computations. In view of possible
generalizations, the general setup of invariant holomorphic principal fibre bundles is
described in a systematic way.
\end{abstract} 

\maketitle
\tableofcontents

\section{Introduction}

The classification of hermitian holomorphic vector bundles, or more general holomorphic principal fibre bundles, over a complex manifold $M$ is a central problem in algebraic geometry and quantization theory, e.g. for a compact Riemann surface $M.$ In geometric quantization, where $M=G/K$ is a co-adjoint orbit, $G$-invariant principal fibre bundles have been investigated from various points of view \cite{Bo,Ra,OR}. In case $M$ is a hermitian symmetric space of non-compact type, a complete characterization of invariant integrable complex structures on principal bundles over $M$ was obtained in \cite{BM} (for the unit disk) and \cite{B} (for all bounded symmetric domains).

The main objective of this paper is to treat the dual case of compact hermitian symmetric spaces, and to show that the compact case (as well as the flat case) leads to exactly the same characterization, resulting in an explicit duality correspondence for invariant integrable complex principal bundles (with hermitian structure) between the non-compact type, the compact type and the flat type as well. The proof is carried out using the Jordan theoretic approach towards hermitian symmetric spaces \cite{L,FK}. Of course, the traditional Lie triple system approach could be used instead, but Jordan triple systems (essentially the hermitian polarization of the underlying Lie triple system) make things more transparent and somewhat more elementary. 

More importantly, the Jordan triple approach leads in a natural way to more general complex homogeneous (non-symmetric) manifolds $G/C$ where $C$ is a proper subgroup of $K.$ These manifolds are fibre bundles over $G/K$ with compact fibres given by Jordan theoretic flag manifolds. Again, there exists a duality between such spaces of compact, non-compact and flat type, and in a subsequent paper \cite{BU} the duality correspondence for invariant fibre bundles, proved here in the hermitian symmetric case $C=K,$ will be studied in the more general setting. In view of these more general situations, the current paper describes the general setup for homogeneous holomorphic principal fibre bundles in a careful way, specializing to the symmetric case only in the last section.
  
As a next step beyond the classification, its dependence on the underlying complex structure on $M$ is of fundamental importance. While the hermitian symmetric case $G/K$ has a unique $G$-invariant complex structure, the more general flag manifold bundles $G/C$ have an interesting moduli space of invariant complex structures. It is a challenging problem whether this moduli space 
carries a projectively flat connexion (with values in the vector bundle of holomorphic sections) similar to the case of abelian varieties, \cite{Mu}, \cite{Wa}, or Chern-Simons theory \cite{ADW}, \cite{Wi}.

From the geometric quantization point of view, it is also of interest to describe the spaces of holomorphic sections, given by suitable Dolbeault operators, in an explicit way.

\section{Homogeneous $H$-bundles}

Let $M$ be a manifold and $G$ a connected real Lie group, with Lie algebra $\gL,$ acting
smoothly on $M.$ Denoting the action $G\xx M\to M$ by $(a,x)\mapsto a(x),$ we define $R_x^M:G\to M$,
$x\in M$, by 
$$R_x^M(a)=a(x)\qquad\forall\ (a,x)\in G\xx M.$$ 
Let $H$ be a connected complex Lie group; its Lie algebra will be denoted by $\hL$. Fix a
maximal compact subgroup $L\ic H;$ all such subgroups are conjugate in $H.$
The Lie algebra of $L$ is denoted by $\Ll.$ Let $Q$ be a $C^\infty$
principal $H$-bundle over $M,$ with projection $\lp:Q\to M=Q/H.$ The free action $Q\xx H\to Q$ is written as $(q,b)\mapsto qb.$ Define $R_b^Q:Q\to Q$ and $L_q^Q:H\to Q$ by 
$$R_b^Qq=L_q^Qb=qb\qquad\forall\ (q,b)\in Q\xx H.$$ 
Then $\lp\circ R_b^Q=\lp$ and $ker(d\lp)_q\subset T_qQ$ is the ``vertical'' subspace of
the tangent space $T_qQ$ at $q\in Q.$ We call $Q$ an {\bf equivariant} $H$-bundle if there
is a $C^\infty$ action $G\xx Q\to Q,$ denoted by $(a,q)\mapsto aq,$ such that $\lp(aq)=a\lp(q)$ and
$$a(qb)=(aq)b.$$
Define $L_a^Q:Q\to Q$ and $R_q^Q:G\to Q$ by 
$$L_a^Q q=R_q^Q a=aq\qquad\forall\ (a,q)\in G\xx Q.$$ 
Then $L_a^Q\circ R_b^Q=R_b^Q\circ L_a^Q$ and $\lp\circ L_a^Q=L_a^M\circ\lp$ for all
$a\in G,\,b\in H.$ A {\bf hermitian structure} on a principal $H$-bundle $Q$ is a 
principal subbundle $P\ic Q$ with structure group $L,$ i.e., we have
$$R_b^Q:P\to P\qquad\forall\ b\in L$$
and the action of $L$ on the fibre $P_x$ is transitive for all $x\in M$.
We also say that $(Q,P)$ is a hermitian $H$-bundle. An equivariant hermitian $H$-bundle
is an equivariant $H$-bundle $Q$ endowed with a hermitian structure $P\ic Q$ such that $G\cdot P=P,$ i.e., we have 
$$L_a^QP=P\qquad\forall\ a\in G.$$ 
In this case $P$ becomes an equivariant $L$-bundle. When $G$ acts transitively on $M$, we
call $Q$ a {\bf homogeneous} $H$-bundle. In the homogeneous case, fix a base point $o\in M$ and put 
$$K=\{k\in G\,\mid\,k(o)=o\}.$$ 
Then $M=G/K.$ The Lie algebra of $K$ will be denoted by $\kL.$

The following proposition is straight-forward to prove.

\begin{proposition}\label{prop1}
Let $f:K\to H$ be a (real-analytic) homomorphism. Consider the quotient manifold 
$$Q:=G\xx_{K,f}H$$ 
consisting of all equivalence classes 
\be{1}\<g|h\>=\<gk^{-1}|f(k)h\>,\ee 
with $g\in G,\,h\in H$ and $k\in K.$ (This bracket notation is used to avoid confusion with the commutator bracket.) Then $Q$ becomes a homogeneous $H$-bundle, with projection 
$\lp(g|h)=g(o)=R_o^M(g).$ The action of $(a,b)\in G\xx H$ is given by
$$a\<g|h\>b=\<ag|hb\>.$$
If in addition, $f(K)\ic L$, we obtain a hermitian homogeneous principal $H$-bundle
$$P:=G\xx_{K,f}L\ic Q:=G\xx_{K,f}H.$$ 
\end{proposition}

In the set-up of Proposition \ref{prop1}, the maps $L_a^Q:Q\to Q$ and $R_{g|h}^Q:G\to Q$
have the form
$$L_a^Q\<g|h\>=R_{g|h}^Q(a)=\<ag|h\>.$$
For any $\el\in H$, let 
$$I_\el^H h=\el h\el^{-1}$$
be the inner automorphism induced by $\el.$ Then $d_eI_\el^H=Ad^{\hL}_\el.$ The conjugate
homomorphism $$I_\el^H\circ f:K\to H$$ induces the $H$-bundle isomorphism
$$G\xx_{K,f}H\to G\xx_{K,I_\el^H\circ f}H$$
mapping the equivalence class $\<g|h\>_f$ to $\<g|I_\el^Hh\>_{I_\el^H\circ f}.$

\begin{theorem}\label{a} Every homogeneous principal $H$-bundle $Q$ on $M= G/K$
is isomorphic to $G\xx_{K,f}H$ for a homomorphism $f:K\to H,$ which is unique up to
conjugation by elements $\el\in H.$ Similarly, every hermitian homogeneous $H$-bundle $(Q,P)$ is isomorphic to the pair
$$G\xx_{K,f}L\ic G\xx_{K,f}H$$
for a homomorphism $f:K\to L\ic H,$ which is unique up to conjugation by an element in $L.$ More precisely, for any base point $\ol\in Q_o$ (respectively, $\ol\in P_o$) there exists a unique homomorphism $f_\ol:K\to H$ (respectively, $f_\ol:K\to L$) such that
\be{2}k\ol=\ol f_\ol(k)\qquad\forall\ k\in K,\ee
and
\be{3}\<g|h\>\mapsto g\ol h\ee 
defines an isomorphism $G\xx_{K,f_\ol}H\to Q$ of equivariant $H$-bundles (respectively, an
isomorphism $G\xx_{K,f_\ol}L\to P$ of hermitian equivariant $H$-bundles). Another base
point $\ol'=\ol\el^{-1},$ with $\el\in H$ (respectively, $\el\in L$), corresponds to the
homomorphism $f_{\ol'}=I_\el\circ f_\ol.$
\end{theorem}

\begin{proof}
Let $Q$ be a homogeneous $H$-bundle. Choose $\ol\in Q$ with $\lp(\ol)=o.$ Since the fibre
$Q_o$ is preserved by $K$, and $H$ acts freely on $Q_o$, there exists a unique map $f_\ol:K\to H$ such that \er{2} holds for all $k\in K.$ Then
$$\ol f_\ol(k_1k_2)\,=\,(k_1k_2)\ol\,=\,
k_1(k_2\ol)\,=\,k_1(\ol f_\ol(k_2))\,=\,(k_1\ol)f_\ol(k_2)
$$
$$
= \, (\ol f_\ol(k_1))f_\ol(k_2)\,=\,
\ol(f_\ol(k_1)f_\ol(k_2))$$ for all $k_1,k_2\in K.$ Since $H$ operates freely on $Q$, this
implies that $f_\ol(k_1k_2)=f_\ol(k_1)f_\ol(k_2)$, and hence $f_\ol:K\to H$ is a (real-analytic) homomorphism. For all $b\in H$ we have
$$k\ol b=(\ol f_\ol(k))b=\ol(f_\ol(k)b).$$
The resulting identity $$(ak)\ol b=a(k\ol b)=a(\ol(f_\ol(k)b))=a\ol(f_\ol(k)b)$$ shows
that \er{3} defines an isomorphism $G\xx_{K,f_\ol}H\to Q$ of equivariant
$H$-bundles, mapping the ``base point'' $\<e|e\>$ to $\ol.$ If $\ol'\in Q_o$ is another
base point, there exists a unique $\el\in H$ such that $\ol'=\ol\el^{-1}.$ It follows that
$$k\ol'=k\ol\el^{-1}=\ol f_\ol(k)\el^{-1}=\ol'\el f_\ol(k)\el^{-1}.$$
Thus the new base point $\ol'$ corresponds to the conjugate homomorphism
$$f_{\ol\el^{-1}}(k)=\el f_\ol(k)\el^{-1}=I_\el^Hf_\ol(k).$$

In the hermitian case, for any base point $\ol\in P_o\ic Q_o$ the defining identity 
\er{2} implies $f_\ol(k)\in L$ for all $k\in K.$ Another base point 
$\ol'=\ol\el^{-1}\in P_o$ differs by a unique element $\el\in L.$ In both cases, since 
$G$ acts transitively on $M,$ the $G$-invariance condition implies that the entire
construction is independent of the choice of base point $o\in M.$
\end{proof}

In view of \er{2}, the homomorphism $f_\ol$ could be denoted by $f_\ol=I_\ol^{-1},$ 
i.e., $k\ol=\ol I_\ol^{-1}(k).$ In this notation, the identity 
$I_{\ol\el^{-1}}^{-1}=I_\el\circ I_\ol^{-1}$ is obvious. It will be convenient to 
express the tangent spaces in an explicit manner using equivalence classes. Note
that the differential $d_ef:\kL\to\hL$ of $f$ at the unit element $e\in K$ is a Lie
algebra homomorphism.

\begin{lemma}
For a given class $\<g|h\>\in Q:=G\xx_{K,f}H$ the tangent space $T_{g|h}Q$ consists of all
equivalence classes
$$\<\d g|\d h\>=\<(d_gR^G_{k^{-1}})\d g-(d_eL^G_{gk^{-1}})\lk|(d_hL^H_{f(k)})\d h+
(d_eR^H_{f(k)h})(d_ef)\lk\>,$$
where $\d g\in T_gG,\,\d h\in T_hH,\,k\in K$ and $\lk\in\kL.$ 
\end{lemma}
Here we regard $TG$ and $TH$ as the disjoint union of the respective tangent spaces, so that the first expression is evaluated at 
$\<g|h\>$ whereas the second expression is evaluated at the same class written as $\<gk^{-1}|f(k)h\>.$ The identity follows from differentiating the relation $\<g_t|h_t\>=\<g_t k_t^{-1}|f(k_t)h_t\>$ at $t=0,$ where $k_t\in K$ satisfies $k_0=k$ and 
$\dl_t^0\,k_t=\lk.$ As special cases we obtain
\be{30}
\<\d g|\d h\>=\<(d_gR^G_{k^{-1}})\d g|(d_hL^H_{f(k)})\d h\>=
\<\d g-(d_eL^G_g)\lk|\d h+(d_eR^H_h)(d_ef)\lk\>\ee
putting $\lk=0$ or $k=e,$ respectively. The projection $\lp$ has the differential 
$$(d_{g|h}\lp)\<\d g|\d h\>= (\frac{d}{dt}(g_t)(0))(o)=(d_gR_o^M)\d g.$$
Thus the vertical subspace is given by 
$$ker(d_{g|h}\lp)=\{\<\d g|\d h\>\, \mid\,\d g\in ker(d_gR_o^M)\ic T_gG,\,\d h\in T_hH\}.$$
For $\lb\in\hL$, the fundamental vector field $\lr_\lb^Q$ has the form
$$(\lr_\lb^Q)_{g|h}=\dl_t^0\,\<g|hb_t\>=\<0_g|(d_eL^H_h)\lb\>=\<0_g|(d_eR^H_h)Ad_h^H\,\lb\>.$$

\section{Connexions and complexions}

A {\bf connexion} on a principal $H$-bundle $Q$ is a smooth distribution $q\mapsto T^\lT_qQ$
of ``horizontal'' subspaces of $T_qQ$ such that $T_qQ=T^\lT_qQ\oplus ker(d_q\lp)$ and
$$T^\lT_{qb}Q=(d_qR_b^Q)(T^\lT_qQ)\qquad\forall\ (q,b)\in Q\xx H.$$
We use the same symbol for the associated connexion 1-form $\lT_q:T_qQ\to\hL$ on $Q,$
uniquely determined by the condition that $X\in T_qQ$ has the horizontal projection 
$$X^\lT=X-(d_eL_q^Q)(\lT_q(X)).$$
A connexion $\lT$ on an equivariant $H$-bundle $Q$ is called {\bf invariant} if
$$(d_qL_a^Q)(T_q^\lT Q)=T_{aq}^\lT Q\qquad\forall\ a\in G.$$
In this case the associated connexion $1$-form satisfies 
$$\lT_q=\lT_{aq}(d_qL_a^Q).$$
Let $Q\xx_{H}\hL$ denote the associated bundle of type $Ad_H,$ with fibres
$$(Q\xx_{H}\hL)_x=\{[q:\lb]=[qh:Ad_h^{-1}\lb]\,\mid\,q\in Q_x,\lb\in\hL\}$$
for $x\in M,$ with $h\in H$ being arbitrary. By \cite[Section II.5]{KN} every tensorial $i$-form on $Q$ is given by
$$(\Cl^Q_q\circ d_q\lp)(X^1_q,\cdots,X^i_q):=\Cl^Q_q((d_q\lp)X^1_q,\cdots,(d_q\lp)X^i_q)$$
for $X^1_q,\cdots,X^i_q\in T_qQ,$ where
$$\Cl_x(v_1\yi\cdots\yi v_i)=[q:\Cl^Q_q(v_1\yi\cdots\yi v_i)]\qquad\forall\ q\in Q_x$$
is an $i$-form of type $Ad_H$ (on $M$), with homogeneous lift $\Cl^Q_q:\W^i T_xM\to\hL$ having the right invariance property
$$\Cl^Q_{qb}=Ad_{b^{-1}}^H\Cl^Q_q\qquad\forall\ b\in H.$$ 
An $i$-form $\Cl$ of type $Ad_H$ is called {\bf invariant} if
\be{4}\Cl^Q_{aq}(d_xL_a^M)=\Cl^Q_q\qquad\forall\ (a,q)\in G\xx Q.\ee
 
\begin{proposition}\label{d}
\mbox{}
\begin{enumerate}
\item[(i)] Let $\lT^0$ be a connexion on $Q$. If $\Cl$ is a $1$-form of type $Ad_H,$ then 
\be{6}\lT_q=\lT^0_q+\Cl^Q_q\circ d_q\lp\qquad\forall\ q\in Q\ee
is a connexion $1$-form on $Q$. Every connexion $1$-form $\lT$ on $Q$ arises this way.\\ 

\item[(ii)] In the equivariant case, let $\lT^0$ be an invariant connexion on $Q.$ Then $\lT$ is invariant if and only if $\Cl$ is invariant, i.e., 
\be{99}\Cl^Q_{aq}(d_xL_a^M)=\Cl^Q_q\qquad\forall\ (a,q)\in G\xx Q.\ee
\end{enumerate}
\end{proposition} 

\begin{proof} Part (i) is well-known. For part (ii) let the connexion
$\lT$ be invariant. The condition $\lp\circ L_a^Q=L_a^M\circ\lp$ implies
that $(d_{aq}\lp)(d_qL_a^Q)=(d_xL_a^M)(d_q\lp)$, and hence we have
$$\Cl^Q_q(d_q\lp)\,=\,\lT_q-\lT_q^0\,=\,(\lT_{aq}-\lT_{aq}^0)(d_qL_a^Q)
$$
$$
=\,\Cl^Q_{aq}(d_{aq}\lp)(d_qL_a^Q)\,=\,\Cl^Q_{aq}(d_xL_a^M)(d_q\lp)\, .$$
Since $d_q\lp$ is surjective, \er{99} follows. The converse is proved in a similar way.
\end{proof}

If \er{6} holds, we say that $\lT$ is related to $\lT^0$ via $\Cl.$ For a hermitian $H$-bundle $(Q,P)$ a connexion $\lX$ on $P$ is called {\bf invariant} if 
$$(d_pL_a^P)T_p^\lX P=T_{ap}^\lX P\qquad\forall\ (a,p)\in G\xx P.$$ 
In this case the associated connexion $1$-form satisfies 
$$\lX_p=\lX_{ap}(d_pL_a^P).$$ 
By \cite[Proposition II.6.2]{KN} every connexion $\lX$ on $P$ has a unique extension to a connexion $\li\lX$ on $Q$ such that
$$T_p^\lX P=T_p^{\li\lX}Q\qquad\forall\ p\in P\ic Q.$$
Equivalently, the connexion forms satisfy
$$(\li\lX)_p|_{T_pP}=\lX_p\qquad\forall\ p\in P.$$
A connexion $\lT$ on $Q$ is called {\bf hermitian} if $\lT=\li\lX$ for a (unique)
connexion $\lX$ on $P.$ Thus hermitian connexions on $Q$ are in 1-1 correspondence with
connexions on $P.$ A connexion $\lX$ on $P$ is invariant if and only if its extension
$\li\lX$ is invariant.

For hermitian $H$-bundles $(P,Q),$ we have $i$-forms of type $Ad_L$ written as
$$\Al_x(v_1\yi\cdots\yi v_i)=[p:\Al^P_p(v_1\yi\cdots\yi v_i)]\qquad\forall\ p\in P_x,$$
with homogeneous lift $\Al^P_p:\W^i T_xM\to\Ll$ having the right invariance property
\be{20}\Al^P_{p\el}=Ad_{\el^{-1}}^L\Al^P_p\qquad\forall\ \el\in L.\ee
Let $\Ll$ be the Lie algebra of $L.$ For each $x\in M$, there is a linear injection of fibres
$$\li_x:(P\xx_L\Ll)_x\to(Q\xx_H\hL)_x,\qquad[p,\lb]\qi[p,\li\lb],$$ 
where $p\in P_x,\lb\in\Ll$ and $\li:\Ll\to\hL$ is the inclusion map. The map $\li_x$ is
well-defined because $$[p_1,\lb_1]=[p_2,\lb_2]\in(P\xx_L\Ll)_x$$ implying
that $p_2=p_1\el$ and $\lb_2=Ad_{\el^{-1}}^L\lb_1$ for some $\el\in L.$ Therefore, we also
have $\li\lb_2=Ad_{\el^{-1}}^H\li\lb_1.$ To show that $\li_x$ is injective, suppose that
$[p_1,\li\lb_1]=[p_2,\li\lb_2]\in(Q\xx_H\hL)_x.$ Then we have $p_2=p_1h$ and $\li\lb_2=Ad_{h^{-1}}^H\li\lb_1$ for some $h\in H.$ Since $p_1,p_2\in P_x$ it follows that $h\in L$ and hence $\lb_2=Ad_{h^{-1}}^L\lb_1.$ 

As a consequence, an $i$-form $\Al$ of type $Ad_L$ induces an $i$-form $\li\Al$ of type $Ad_H,$ with homogeneous lift
$$(\li\Al)_{ph}^Q:=Ad_{h^{-1}}^H\Al_p^P:\W^i T_oM\to\hL\qquad\forall\ (p,h)\in P\xx H.$$ 
Then $\Al$ is invariant in the sense that
$$\Al^P_{ap}(d_xL_a^M)=\Al^P_p\qquad\forall\ (a,p)\in G\xx P$$ 
if and only if $\li\Al$ is invariant as in \er{4}.

\begin{proposition}\label{e}
\mbox{}
\begin{enumerate}
\item[(i)] For a hermitian bundle $(P,Q)$, let $\lX^0$ be a connexion on $P$. If $\Al$ is a $1$-form of type $Ad_L,$ then 
\be{18}\lX_p=\lX^0_p+\Al^P_p\circ d_p\lp\qquad\forall\ p\in P\ee
is a connexion $1$-form on $P$. Every connexion $1$-form $\lX$ on $P$ arises this way. We also have
$$(\li\lX)_q=(\li\lX)^0_q+(\li\Al)^Q_q\circ d_q\lp\qquad\forall\ q\in Q.$$

\item[(ii)] In the equivariant case, let $\lX^0$ be an invariant connexion on $P.$ Then $\lX$ is invariant if and only if $\Al$ is invariant, i.e., 
\be{5}\Al^P_{ap}(d_xL_a^M)=\Al^P_p\qquad\forall\ (a,p)\in G\xx P.\ee
\end{enumerate}
\end{proposition}

If \er{18} holds, we say that $\lX$ is related to $\lX^0$ via $\Al.$ In this case, $\li\lX$ is related to $\li\lX^0$ via $\li\Al.$ 
Now suppose that $(M,j)$ is an almost complex manifold, with almost complex structure $j_x\in\,End(T_xM),\,x\in M,$ having the left invariance property
$$(d_xL_g^M)j_x=j_{gx}(d_xL_g^M)\qquad\forall\ (g,x)\in G\xx M.$$ 
Let $H$ be a complex Lie group. Consider the bi-invariant complex structure $i_h\in End(T_hH)$ on $H$ such that $i_e\in End(\hL)$ is multiplication by $\F{-1}.$
See \cite{A}, \cite{K} for complex structures on principal bundles.

\begin{definition} An almost complex structure $J_q\in End\,(T_qQ)$ on an $H$-bundle $Q$ is called a {\bf complexion} if
\be{7}(d_q\lp)J_q=j_x(d_q\lp)\ee
and the map $Q\xx H\to Q$ is almost-holomorphic. Writing $qb=L_q^Q(b)=R_b^Q(q)$ for
$q\in Q,\,b\in H,$ this means that
\be{8}(d_qR_b^Q)J_q=J_{qb}(d_qR_b^Q)\ee
\be{9}J_{qb}(d_bL_q^Q)=(d_bL_q^Q)i_b .\ee
In the equivariant case, a complexion $J$ is called {\bf invariant} if in addition
$$(d_qL_a^Q)J_q=J_{aq}(d_qL_a^Q)\qquad\forall\ a\in G.$$
\end{definition}
By right invariance the condition \er{9} is equivalent to
$$J_q(d_eL_q^Q)\lb=(d_eL_q^Q)(\F{-1}\lb)\qquad\forall\ \lb\in\hL.$$ 
Since the fibre $(Q\xx_H\hL)_x$ is a complex vector space, the
notion of $(p,q)$-forms of type $Ad_H$ makes sense.

\begin{proposition}\label{i}
\mbox{}
\begin{enumerate}
\item[(i)] Let $J^0$ be a complexion on $Q.$ If $\Bl$ is a $(0,1)$-form of type $Ad_H,$ then
\be{11}J_q=J_q^0+(d_eL_q^Q)(\Bl^Q_q\circ d_q\lp)\ee
defines a complexion $J$ on $Q.$ Every complexion $J$ on $Q$ arises this way.\\

\item[(ii)] Let $J^0$ be an invariant complexion on an equivariant bundle $Q.$ Then $J$ is
invariant if and only if $\Bl$ is invariant, i.e.,
\be{10}\Bl^Q_{aq}(d_xL_a^M)=\Bl^Q_q\qquad\forall\ (a,q)\in G\xx Q.\ee 
\end{enumerate}
\end{proposition}

\begin{proof} By \er{7} we have $$(d_q\lp)(J_q-J_q^0)\,=\,(j_x-j_x)(d_q\lp)
\,=\, 0\, .$$ Thus
$Image(J_q-J_q^0)\ic Ker(d_q\lp).$ Since
$d_eL_q^Q:\hL\to Ker(d_q\lp)$ is an isomorphism, there exists a $1$-form $\lQ_q:T_qQ\to\hL$ such that
$$J_q-J_q^0=(d_eL_q^Q)\lQ_q.$$
For $b\in H$ we have $R_b^Q\circ L_q^Q=L_{qb}^Q\circ I_{b^{-1}}^H$, and hence $(d_qR_b^Q)(d_eL_q^Q)=(d_eL_{qb}^Q)Ad_{b^{-1}}^H.$ It follows that
$$(d_eL_{qb}^Q)\lQ_{qb}(d_qR_b^Q)\,=\,(J_{qb}-J_{qb}^0)(d_qR_b^Q)\,=\,
(d_qR_b^Q)(J_q-J_q^0)
$$
$$
=\, (d_qR_b^Q)(d_eL_q^Q)\lQ_q
\, =\, (d_eL_{qb}^Q)Ad_{b^{-1}}^H\lQ_q\, .$$
Therefore, we have $\lQ_{qb}(d_qR_b^Q)=Ad_{b^{-1}}^H\lQ_q,$ so $\lQ$ is pseudo-tensorial. For $\lb\in\hL$ we have
$$J_q(d_eL_q^Q)\lb=(d_eL_q^Q)(\F{-1}\lb)=J_q^0(d_eL_q^Q)\lb.$$
Therefore, we have $(J_q-J_q^0)|_{Ker(d_q\lp)}=0$, which implies that $\lQ$ is tensorial. Hence there exists a unique $1$-form 
$\Bl$ of type $Ad_H$ such that $\lQ_q=\Bl^Q_q(d_q\lp).$ We also have $(J_q-J_q^0)^2=0$ and hence
$$(d_eL_q^Q)(\Bl^Q_q\,j_x+\F{-1}\Bl^Q_q)(d_q\lp)=(d_eL_q^Q)\Bl^Q_q\,j_x(d_q\lp)+(d_eL_q^Q)\F{-1}\Bl^Q_q(d_q\lp)$$
$$=(d_eL_q^Q)\Bl^Q_q(d_q\lp)J_q^0+J_q^0(d_eL_q^Q)\Bl^Q_q(d_q\lp)=(J_q-J_q^0)J_q^0+J_q^0(J_q-J_q^0)$$
$$=J_q^2-(J_q-J_q^0)^2-(J_q^0)^2=J_q^2-(J_q^0)^2=-id+id=0.$$
Since $d_eL_q^Q$ is invertible, we obtain
$$\Bl^Q_qj_x+\F{-1}\Bl^Q_q=0.$$
Thus $\Bl$ is of type $(0,1)$.

For part (ii), let $J$ be invariant. Using $L_a^Q\circ L_q^Q=L_{aq}^Q$ and $L_a^M\circ\lp=\lp\circ L_a^Q$ we obtain
$$(d_eL_{aq}^Q)\Bl^Q_{aq}(d_xL_a^M)(d_q\lp)=(d_eL_{aq}^Q)\Bl^Q_{aq}(d_{aq}\lp)(d_qL_a^Q)=(J_{aq}-J_{aq}^0)(d_qL_a^Q)$$
$$=(d_qL_a^Q)(J_q-J_q^0)=(d_qL_a^Q)(d_eL_q^Q)\Bl^Q_q(d_q\lp)=(d_eL_{aq}^Q)\Bl^Q_q(d_q\lp).$$
Since $d_eL_{aq}^Q$ is invertible, \er{10} follows. The converse is proved in a similar way.
\end{proof}
If \er{11} holds, we say that $J$ is related to $J^0$ via $\Bl.$ There is a close relationship between (invariant) connexions and complexions: Every connexion $\lT$ on $Q$ induces a unique complexion $J^\lT$ which is ``horizontal'' in the sense that
$$J_q^\lT:T^\lT_qQ\to T^\lT_qQ\qquad\forall\ q\in Q.$$ 
In fact, for vertical tangent vectors we have 
\be{31}J^\lT_q(d_eL_q^Q)\lb=(d_eL_q^Q)(\F{-1}\lb)\qquad\forall\lb\in\hL,\ee 
and if $Y\in T^\lT_qQ$ is horizontal, then $J^\lT_q(Y)\in T^\lT_qQ$ is uniquely
determined by the condition $(d_q\lp)J^\lT_q(Y)=j_x(d_q\lp)(Y).$ If $\lT$ is invariant, then the associated complexion $J^\lT$ is also invariant. 

\begin{proposition}\label{b}
\mbox{}
\begin{enumerate}
\item[(i)] If a connexion $\lT$ is related to $\lT^0$ via $\Cl,$ then the induced complexion $J^\lT$ is related to $J^0$ via $\Bl,$ where
\be{12}\Bl_x=\sqrt{-1}\Cl_x-\Cl_x\,j_x\qquad\forall\ x\in M.\ee 

\item[(ii)] $\lT$ and $\lT^0$ induce the same complexion if and only if $\Cl$ is a $(1,0)$-form, i.e., $\Cl^Q_q:T_xM\to\hL$ is $\mathbb C$-linear.\\

\item[(iii)] If $J^0$ is induced by a connexion $\lT^0$ and $J$ is related to $J^0$ via
$\Bl,$ then $J$ is induced by the connexion 
$$\lT_q=\lT^0_q-\f{\sqrt{-1}}2\Bl^Q_q(d_q\lp)$$
which is related to $\lT^0$ via $-\f{\sqrt{-1}}2\Bl.$
\end{enumerate}
\end{proposition}

\begin{proof} In order to check \er{11} on a tangent vector $X\in T_qQ,$ we may assume that $X$ is $\lT^0$-horizontal, since for vertical vectors of the form $X=(d_eL_q^Q)\lb,\,\lb\in\hL,$ both sides of \er{11} agree. Thus assume that $\lT_q^0X=0.$ 
Then $\lT_qX=\Cl^Q_q(d_q\lp)X.$ Since $J^0$ is induced by $\lT^0,$ it follows that $J_q^0X$ is also $\lT^0$-horizontal, and hence $\lT_qJ_q^0X=\Cl^Q_q(d_q\lp)J_q^0X.$
Therefore, $$X-(d_eL_q^Q)\Cl^Q_q(d_q\lp)X=X-(d_eL_q^Q)\lT_qX$$ is $\lT$-horizontal. It follows that 
$$U:=J^\lT_qX-(d_eL_q^Q)i\Cl^Q_q(d_q\lp)X=J^\lT_q\(X-(d_eL_q^Q)\Cl^Q_q(d_q\lp)X\),$$ 
$$V:=J_q^0X-(d_eL_q^Q)\Cl^Q_q(d_q\lp)J_q^0X=J_q^0X-(d_eL_q^Q)\lT_qJ_q^0X$$ 
are both $\lT$-horizontal. Since $(d_q\lp)U=j_x(d_q\lp)X=(d_q\lp)V$, it follows that
$U=V$, i.e., 
$$J^\lT_qX-(d_eL_q^Q)\sqrt{-1}\Cl^Q_q(d_q\lp)X\,=\,
J_q^0X-(d_eL_q^Q)\Cl^Q_q(d_q\lp)J_q^0X\, .$$
Therefore, we have
$$(d_eL_q^Q)\Bl^Q_q(d_q\lp)X\,=\,J^\lT_qX-J_q^0X
$$
$$
=\, (d_eL_q^Q)\(\sqrt{-1}\Cl^Q_q(d_q\lp)X-
\Cl^Q_q(d_q\lp)J_q^0X\)\, .$$
Since $d_eL_q^Q$ is injective, using \er{7} it follows that
$$\Bl^Q_q(d_q\lp)X=\sqrt{-1}\Cl^Q_q(d_q\lp)X-\Cl^Q_q(d_q\lp)J_q^0X=
(\sqrt{-1}\Cl^Q_q-\Cl^Q_qj_x)(d_q\lp)X.$$
We now have $\Bl^Q_q=\sqrt{-1}\Cl^Q_q-\Cl^Q_qj_x$ on $T_xM$ because
$d_q\lp$ is surjective. This proves part (i).

Part (ii) is a direct consequence, since $\Bl=0$ if and only
if $\Cl$ is $\mathbb C$-linear.

For part (iii), since $\Bl^Q_q:T_xM\to\hL$ is
$\mathbb C$-antilinear, the $1$-form $\Cl_x:=-\f{\sqrt{-1}}2\,\Bl_x$ yields $\Bl_x$ via \er{12}.
\end{proof}

In the hermitian case, we can sharpen the correspondence as follows:

\begin{proposition}\label{j}
\mbox{}
\begin{enumerate}
\item[(i)] For a hermitian $H$-bundle $(Q,P),$ a complexion $J$ on $Q$ is induced by
a unique hermitian connexion $\li\lX,$ with $\lX$ being a connexion on $P.$ Thus there is a
$1-1$ correspondence between complexions on $Q,$ connexions
on $P,$ and hermitian connexions on $Q$.

\item[(ii)] In the equivariant case, $J$ is invariant if and only if $\lX$
(equivalently, $\li\lX$) is invariant.
\end{enumerate}
\end{proposition}

\begin{proof} Let $\lX^0$ be a connexion on $P$, and let $J^0$ be the complexion on $Q$
induced by $\li\lX^0.$ The complexion $J$ is related to $J^0$ via a unique $(0,1)$-form
$\Bl$ of type $Ad_H.$ Since the Lie algebra splitting
\be{13}\hL=\Ll\oplus\,\sqrt{-1}\Ll\ee 
is $Ad_L$-invariant, the anti-linear map $\Bl_x:T_xM\to(Q\xx_H\hL)_x$ has a unique decomposition
\be{19}\Bl_xv=\sqrt{-1}(\li_x\Al_x)v-(\li_x\Al_x)j_xv=\sqrt{-1}(\li\Al)_xv-(\li\Al)_xj_xv\qquad\forall\ x\in M,\ee
where $\Al_x:T_xM\to(P\xx_L\Ll)_x$ is a $1$-form of type $Ad_L.$ Let $\lX$ be the
connexion on $P$ related to $\lX^0$ via $\Al.$ By Proposition \ref{b},
the complexion $J$ is induced by the hermitian connexion 
$$(\li\lX^0)_q+(\li\Al)^Q_q\circ d_q\lp=(\li\lX)_q.$$ 

For uniqueness, suppose that for two connexions $\lX^0$ and $\lX$ on $P$, the extensions induce the same complexion. The connexion 
$\lX$ is related to $\lX^0$ via a unique $1$-form $\Al$ of type $Ad_L.$ Then $\li\lX$ is
related to $\li\lX^0$ via $\li\Al$, and both induce the same complexion. Proposition
\ref{b} implies that $\li\Al$ is a $(1,0)$-form of type $Ad_H.$ Since $\Al^P_p$ is
$\Ll$-valued, it follows that $\Al_p^P=0.$ Hence $\Al=0$ and $\lX=\lX^0.$ As $\lX$ is invariant if and only if the complexion induced by $\li\lX$ is invariant, the assertion (ii) follows.
\end{proof}

\section{Homogeneous connexions and complexions}

We now turn to the homogeneous case. By Theorem \ref{a} we may assume that the
homogeneous $H$-bundle $Q$ is given by $Q=G\xx_{K,f}H$ for a homomorphism $f:K\to H$,
with $f(K)\ic L$ in the hermitian case. We fix an $Ad_K$-invariant splitting
\be{14}\gL=\kL\oplus\mL.\ee
Thus $[\kL,\mL]\ic\mL$ but not necessarily $[\mL,\mL]\ic\kL.$ Let $P_\kL:\gL\to\kL$ and
$P_\mL:\gL\to\mL$ denote the projection maps given by \eqref{14}.

\begin{definition}
\mbox{}
\begin{enumerate}
\item[(i)] The {\bf tautological} connexion $\lT^0$ on $Q=G\xx_{K,f}H$ is defined by the horizontal subspaces
$$T_{g|h}^0Q\,:=\, \{\<(d_eL^G_g)\Lt|0_h\>:\,\Lt\in\mL\}
$$
$$
=\, \{\<(d_eL^G_g)\lg|(d_eR^H_h)\lh\>:\,\lg\in\gL,\,\lh\in\hL,\,(d_ef)P_\kL\lg+\lh=0\}$$
$$=\{\<\d g|\d h\>:\,\d g\in T_gG,\,\d h\in T_hH,\,
(d_ef)P_\kL(d_eL^G_g)^{-1}\d g+(d_eR^H_h)^{-1}\d h=0\}.$$

\item[(ii)] For a hermitian homogeneous $H$-bundle $(Q,P)$, the tautological connexion is
the extension of the ``tautological'' connexion $\lX^0$ on $P$ with horizontal subspaces
$$T_{g|\el}^0P:=\{\<(d_eL^G_g)\Lt|0_\el\>\,\mid\,\Lt\in\mL\}\qquad\forall\ \<g|\el\>\in P=G\xx_{K,f}L.$$
\end{enumerate}
\end{definition}

The equivalence of the various realizations of the horizontal subspaces follows from the definition. Under right translations, we have
$$(d_{g|h}R_b^Q)\<\d g|\d h\>=\dl_t^0\,\<g_t|h_t b\>=\<\d g|(d_hR^H_b)\d h\>.$$
It follows that $(T_{g|h}R_b^Q)T_{g|h}^0Q=T_{\<g|hb\>}^0Q,$ showing that $\lT^0$ is indeed a connexion on $Q.$ Since
$$(d_{g|h}L_a^Q)\<\d g|\d h\>=\<(d_gL^G_a)\d g|\d h\>,$$ it follows that
$$(d_{g|h}L_a^Q)\<(d_eL^G_g)\Lt|0_h\>=\<(d_gL^G_a)(d_eL^G_g)\Lt|0_h\>=(d_eL^G_{ag}\Lt|0_h\>.$$
Thus $\lT^0$ is an invariant connexion. In the hermitian case, the connexion $\lX^0$ on $P$ is also invariant.

\begin{proposition}\label{c}
The connexion $1$-form of the tautological connection is given by 
$$\lT^0_{g|h}\<(d_eL^G_g)\lg|(d_eL_h^H)\lh\>=Ad_{h^{-1}}^H(d_ef)P_\kL\lg+\lh\qquad\forall\ (\lg,\lh)\in\gL\xx\hL.$$ 
In particular, $\lT^0_{e|e}\<\lg|\lh\>=(d_ef)P_\kL\lg+\lh.$
\end{proposition}

\begin{proof} Applying \er{30} to $\lk=P_\kL\lg$ we conclude that
$$\<(d_eL^G_g)\lg|(d_eL_h^H)\lh\>=\<(d_eL^G_g)P_\mL\lg|(d_eR^H_h)(d_ef)P_\kL\lg+(d_eL_h^H)\lh\>.$$
$$=\<(d_eL^G_g)P_\mL\lg|0_h\>+\<0_g|(d_eR^H_h)(d_ef)P_\kL\lg+(d_eL_h^H)\lh\>.$$
This gives the direct sum decomposition into horizontal and vertical components, and 
$$\lb:=\lT^0_{g|h}\<(d_eL^G_g)\lg|(d_eL_h^H)\lh\>\in\hL$$ is determined by the condition
$$\<0_g|(d_eR^H_h)(d_ef)P_\kL\lg+(d_eL_h^H)\lh\>\,=\,(d_eL_{g|h})\lb
$$
$$
=\, \dl_t^0\,\<g|h\>b_t\,=\,\dl_t^0\,\<g|hb_t\>\,=\,\<0|(d_eL_h^H)\lb\>\, .$$
Therefore, we have $\lb=(d_eL_h^H)^{-1}(d_eR^H_h)(d_ef)P_\kL\lg+\lh=Ad_{h^{-1}}^H(d_ef)P_\kL\lg+\lh.$
\end{proof} 

In the homogeneous case, invariance can be checked at the base point $o\in M.$ An $i$-linear map $\cl:\W^i T_oM\to\hL$ will be called $f$-{\bf covariant} if 
$$Ad_{f(k)}^H\circ\cl=\cl\circ(d_oL_k^M)\qquad\forall\ k\in K.$$ 
Every such map induces an invariant $i$-form $\Cl$ of type $Ad_H$ via
$$\Cl_{go}(d_oL_g^M)v=[\<g|h\>,Ad_h^{-1}\cl v]$$
for all $(g,h)\in Q=G\xx_{K,f}H$ and $v\in\W^i T_oM.$ Equivalently, the homogeneous lift satisfies
\be{16}\Cl^Q_{g|h}(d_oL_g^M)=Ad_{h^{-1}}^H\circ\cl.\ee
Conversely, for any invariant $i$-form $\Cl$ of type $Ad_H$, the map $\cl:=\Cl_{e|e}^Q:\W^i T_oM\to\hL$ is $f$-covariant. If the homomorphism $f$ is replaced by the conjugate $f'=I_\el\circ f,$ then $\cl':\W^i T_oM\to\hL$ is $f'$-covariant if and only if $\cl:=Ad_{\el^{-1}}^H\circ\cl'$ is $f$-covariant. 
In the homogeneous case Proposition \ref{d} yields the following:

\begin{proposition}\label{f}
Let $\cl:T_oM\to\hL$ be a linear map which is $f$-covariant,
i.e., $$Ad_{f(k)}^H\circ\cl=\cl\circ(d_oL_k^M)$$ for all $k\in K.$ Then the $1$-form $\Cl$ of type $Ad_H$ defined by 
$$\Cl^Q_{g|h}(d_oL_g^M)=Ad_{h^{-1}}^H\circ\cl,\qquad\forall\ (g,h)\in G\xx H$$ 
is invariant. Hence the connexion $\lT$ on $Q=G\xx_{K,f}H,$ related to the tautological
connexion $\lT^0$ via $\Cl,$ is invariant, and also every invariant connexion is of this form.
\end{proposition}

In this case we say that $\lT$ is generated by $\cl,$ the choice of the (invariant) 
tautological connexion $\lT^0$ being understood. In the hermitian homogeneous case 
Proposition \ref{e} yields the following:

\begin{proposition}\label{g} Let $\al:T_oM\to\Ll$ be a linear map which is $f$-covariant,
i.e., satisfies $Ad_{f(k)}^L\circ\al=\al\circ(d_oL_k^M)$ for all $k\in K.$ Then the $1$-form $\Al$ of type $Ad_L$ defined by 
$$\Al^Q_{g|\el}(d_oL_g^M)=Ad_{\el^{-1}}^L\circ\al\qquad\forall\ (g,\el)\in G\xx L,$$ 
is invariant. Hence the connexion $\lX$ on $P=G\xx_{K,f}\Ll$, related to the tautological connexion $\lX^0$ via $\Al,$ is invariant, and every invariant connexion is of this form.
\end{proposition}

In this case we say that $\lX$ is generated by $\al,$ the choice of the (invariant) tautological connexion $\lX^0$ being understood. Since $\li\lX^0=\lT^0,$ the invariant connexion $\li\lX$ is generated by the $f$-covariant map $\li\circ\al:T_oM\to\hL.$
Combining Propositions \ref{f} and \ref{g} with Theorem \ref{a} we obtain the following:

\begin{theorem}\label{h}
\mbox{}
\begin{enumerate}
\item[(i)] The homogeneous $H$-bundles $Q$ endowed with an invariant connexion are in $1-1$ 
correspondence with pairs $(f,\cl)$ consisting of a homomorphism $f:K\to H$ and an 
$f$-covariant linear map $\cl:T_oM\to\hL,$ modulo the equivalence 
$(f,\cl)\os(I_\el^H\circ f,Ad_\el^H\,\cl),$ where $\el\in H$ is arbitrary.\\

\item[(ii)] The hermitian homogeneous $H$-bundles $(P,Q)$ endowed with an invariant hermitian 
connexion are in $1-1$ correspondence with pairs $(f,\al)$ consisting of a 
homomorphism $f:K\to L$ and an $f$-covariant linear map $\al:T_oM\to\Ll,$ modulo the 
equivalence $(f,\al)\os(I_\el^L\circ f,Ad_\el^H\,\al),$ where $\el\in L$ is arbitrary.

\end{enumerate}
\end{theorem}

In the complex homogeneous case, Proposition \ref{i} yields the following:

\begin{proposition}\label{k}
Let $\bl:T_oM\to\hL$ be an anti-linear map which is $f$-covariant, i.e.,
$$Ad_{f(k)}^H\circ\bl=\bl\circ(d_oL_k^M),\qquad \forall\ k\in K.$$
Then the $(0,1)$-form $\Bl$ of type $Ad_H$ defined by 
$$\Bl^Q_{g|h}(d_oL_g^M)=Ad_{h^{-1}}^H\circ\bl\qquad\forall\ (g,h)\in G\xx H$$ 
is invariant. Hence the complexion $J$ on $Q=G\xx_{K,f}H,$ related to the tautological complexion $J^0$ via $\Bl,$ is invariant, and every invariant complexion $J$ is of this form.
\end{proposition}

In this case we say that $J$ is generated by $\bl,$ the choice of the (invariant)
tautological complexion $J^0$ being understood.
Using the $K$-invariant inverse map $T_oM\ni v\mapsto\t v\in\mL$ of $d_eR_o^M,$ the
tautological complexion $J^0$ has the value
$$J^0_{e|e}(\<\t v|0_e\>)=\<(j_o\,v)^\os|0_e\>$$
at the base point. For a general invariant complexion $J$ generated by $\bl:T_oM\to\hL$ we
have $$L_{e|e}^Qh=\<e|e\>h=\<e|h\>$$ and hence $d_{e|e}L_{e|e}^Q\lh=\<0_e|\lh\>.$ Therefore 
$$J_{e|e}(\<\t v|0_e\>)-J^0_{e|e}(\<\t v|0_e\>)\,=\,(d_{e|e}L_{e|e}^Q)\bl(d_{e|e}\lp)
(\<\t v|0_e\>)
$$
$$
=\, (d_{e|e}L_{e|e}^Q)\bl v=\<0_e|\bl\,v\>$$ 
and hence
$$J_{e|e}(\<\t v|0_e\>)=\<(j_o\,v)^\os|0_e\>+\<0_e|\bl\,v\>=\<(j_o\,v)^\os|\bl\,v\>.$$
All other values can be computed via $G\xx H$-invariance. In the complex homogeneous case Proposition \ref{b} yields the following:

\begin{proposition}\label{l}
\mbox{}
\begin{enumerate}
\item[(i)] If an invariant connexion $\lT$ is generated by $\cl$, then the induced invariant complexion $J^\lT$ is generated by
\be{25}\bl=\sqrt{-1}\cl-\cl\,j_o.\ee 

\item[(ii)] $\lT$ induces the tautological complexion if and only if $\cl:T_oM\to\hL$ is
$\mathbb C$-linear.\\

\item[(iii)] If the invariant complexion $J$ is generated by $\bl,$ then $J$ is induced
by the invariant connexion $\lT$ generated by $-\f{\sqrt{-1}}2\bl.$
\end{enumerate}
\end{proposition}

In the hermitian homogeneous case. there is a 1-1 correspondence \be{21}\bl 
v=\sqrt{-1}\al (v)-\al\,j_ov\qquad\forall\ v\in T_oM\ee between $f$-covariant 
anti-linear maps $\bl:T_oM\to\hL$ and $f$-covariant maps $\Al:T_oM\to\Ll,$ since $j_o$ 
commutes with $d_oL_k^M$ for all $k\in K.$ This correspondence realizes the 
correspondence between invariant complexions and invariant hermitian connexions 
addressed in Proposition \ref{j}. Combining Propositions \ref{k} and \ref{l} with 
Theorem \ref{a}, we obtain the following:

\begin{theorem}
The homogeneous $H$-bundles $Q$ (respectively, the hermitian homogeneous $H$-bundles 
$(P,Q)$) endowed with an invariant complexion $J$ are in $1-1$ correspondence with 
pairs $(f,\bl)$ consisting of a homomorphism $f:K\to H$ (respectively, $f:K\to L$) and 
an $f$-covariant anti-linear map $$\bl:T_oM\to\hL,$$ modulo the equivalence 
$(f,\bl)\os(I_\el^H\circ f,Ad_\el^H\,\bl),$ where $\el\in H$ (respectively, $\el\in 
L$) is arbitrary. In the hermitian case, we may equivalently consider pairs $(f,\al)$ 
where $\al:T_oM\to\Ll$ is any $f$-covariant $\mathbb R$-linear map.
\end{theorem}

\section{Curvature and Integrability}

For any connexion $\lT$ on a principal $H$-bundle $Q$ there exists a $2$-form $\Kl$ of type $Ad_H$ giving the {\bf curvature}
$$d\lT(X,Y)+\f12[\lT X,\lT Y]=\Kl^Q_q((d_q\lp)X,(d_q\lp)Y)\qquad\forall\ X,Y\in T_qQ.$$ 
If $\lT$ is invariant, the associated curvature form $\Kl$ is also invariant. For $\lb\in\hL$ define the right action vector field $\lr_\lb^Q$ on $Q$ by
$$(\lr_\lb^Q)_q=\dl_t^0q\exp(t\lb)=(d_eL_q^Q)\lb\qquad\forall\ q\in Q.$$ 
Let $J$ be a complexion on $Q.$ For vector fields $X,Y$ the $(0,2)$-part of the bilinear bracket $[X,Y]$ defined by 
$$N(X,Y)=[X,Y]+J[JX,Y]+J[X,JY]-[JX,JY]$$
is called the {\bf Nijenhuis tensor}. It is well-known that $J$ is integrable if and only if $N$ vanishes. 

\begin{lemma}\label{m}
Let $J$ be a complexion on $Q.$ Then $N(\lr_\lb^Q,Y)=0$ for $\lb\in\hL$ and any vector field $Y$ on $Q.$ 
\end{lemma}

\begin{proof} Let $$b_t\,=\,\exp(t\lb)\, .$$ Applying \cite[Proposition I.1.9]{KN} to $\lr_\lb^Q=\dl_t^0\exp(R_{b_t}^Q)$ we have
$$[\lr_\lb^Q,Y]_q=\lim_{t\to 0}\,(Y_q-(d_{qb_{-t}}R_{b_t}^Q)Y_{qb_{-t}})=\lim_{t\to 0}\,(Y_q-(d_{qb_t}R_{b_{-t}}^Q)Y_{qb_t})$$
for every vector field $X$ on $Q.$ Since $J$ commutes with right translations $R_b^Q$ on $Q$ it follows that
$$J_q[\lr_\lb^Q,Y]_q\,=\,\lim_{t\to 0}\,(J_qY_q-J_q(d_{qb_t}R_{b_{-t}}^Q)Y_{qb_t})
$$
$$
=\,\lim_{t\to 0}\,(J_qY_q-(d_{qb_t}R_{b_{-t}}^Q)J_{qb_t}Y_{qb_t})\,=\,[\lr_\lb^Q,JY]_q\, .$$
Thus $J[\lr_\lb^Q,Y]=[\lr_\lb^Q,JY]$ as vector fields. Since $JJY=-Y$ and $J\lr_\lb^Q=\lr_{\F{-1}\lb}^Q,$ we obtain
$$N(\lr_\lb^Q,Y)=[\lr_\lb^Q,Y]+J[\lr_\lb^Q,JY]+J[\lr_{\F{-1}\lb}^Q,Y]-[\lr_{\F{-1}\lb}^Q,JY]$$
$$=[\lr_\lb^Q,Y]+[\lr_\lb^Q,JJY]+[\lr_{\F{-1}\lb}^Q,JY]-[\lr_{\F{-1}\lb}^Q,JY]=0.$$
\end{proof}

\begin{theorem}\label{r}
Let $(M,j)$ be integrable. Then the complexion $J^\lT$ has the Nijenhuis tensor
$$N_q(X,Y)\,=\,-2(d_eL_q^Q)\o\Kl_q^Q((d_q\lp)X,(d_q\lp)Y)\qquad \forall\ X,Y\in T_qQ\, ,$$ 
where the $(0,2)$-part $\o\Kl$ of $\Kl$ is defined by
$$\o\Kl_x(u,v)
$$
$$
=\, \Kl_x(u,v)+\sqrt{-1}\Kl_x(j_xu,v)+\sqrt{-1}\Kl_x(u,j_xv)-\Kl_x(j_xu,j_xv)\ \ \forall\ u,v\in T_xM.$$
\end{theorem} 

\begin{proof}
Every $X\in T_qQ$ is given by $X=(\lx^\lT+\lr_\lb^Q)_q$ for some vector field $\lx$ on $M$ and $\lb\in\hL.$ Thus it suffices to consider vector fields of the form $\lx^\lT+\lr_\lb^Q.$ By Lemma \ref{m} it is enough to consider horizontal lifts of vector fields $\lx,\lh$ on $M.$ Denoting the Nijenhuis tensor of $(M,j)$ by $n(\lx,\lh),$ the integrability
assumption on $(M,j)$ implies that
\be{22}n(\lx,\lh)=[\lx,\lh]+j[j\lx,\lh]+j[\lx,j\lh]-[j\lx,j\lh]=0.\ee
{}From \cite[Corollary I.5.3]{KN} we have 
$$\lT_q[\lx^\lT,\lh^\lT]_q=-2\lO_q(\lx_q^\lT,\lh_q^\lT)=-2\Kl^Q_q(\lx_x,\lh_x).$$
By \cite[Proposition I.1.3]{KN} the horizontal part of $[\lx^\lT,\lh^\lT]_q$ coincides with $[\lx,\lh]_q^\lT.$ It follows that
$$[\lx^\lT,\lh^\lT]_q=[\lx,\lh]^\lT_q+(d_eL_q^Q)\lT_q[\lx^\lT,\lh^\lT]_q=[\lx,\lh]^\lT_q-2(d_eL_q^Q)\Kl^Q_q(\lx_x,\lh_x).$$
Using \er{31} and $J^\lT\lx^\lT=(j\lx)^\lT$ this implies that
$$J^\lT_q[\lx^\lT,\lh^\lT]_q\,=\,J^\lT_q[\lx,\lh]^\lT_q-2J^\lT_q(d_eL_q^Q)\Kl^Q_q(\lx_x,\lh_x)
$$
$$
=\,(j[\lx,\lh])^\lT_q-2(d_eL_q^Q)\sqrt{-1}\Kl^Q_q(\lx_x,\lh_x)$$
and $[J^\lT\lx^\lT,\lh^\lT]_q=[j\lx,\lh]^\lT_q-2(d_eL_q^Q)\Kl^Q_q(j\lx_x,\lh_x).$ Therefore,
we have
$$N_q(\lx_q^\lT,\lh_q^\lT)=[\lx^\lT,\lh^\lT]_q+J^\lT_q[J^\lT\lx^\lT,\lh^\lT]_q+J^\lT_q[\lx^\lT,J^\lT\lh^\lT]_q
-[J^\lT\lx^\lT,J^\lT\lh^\lT]_q$$
$$=(n(\lx,\lh))^\lT_q-2(d_eL_q^Q)\(\Kl^Q_q(\lx_x,\lh_x)+\sqrt{-1}\Kl^Q_q(j\lx_x,\lh_x)+
\sqrt{-1}\Kl^Q_q(\lx_x,j\lh_x)
$$
$$
-\Kl^Q_q(j\lx_x,j\lh_x)\)\,=\,(n(\lx,\lh))^\lT_q-2(d_eL_q^Q)\o\Kl_q^Q(\lx_x,\lh_x)\, .$$
In view of \er{22} the proof is now complete.
\end{proof}

\begin{corollary}\label{p}
If $(M,j)$ is integrable, then the complexion $J^\lT$ is integrable on $Q$ if and only if the curvature form $\Kl$ of 
$\lT$ has vanishing $(0,2)$-part. 
\end{corollary}

In the homogeneous case the curvature form $\Kl$ of an invariant connexion $\lT$ is induced by a unique $f$-covariant map $\kl:\W^2 T_oM\to\hL$ as in \er{16}.

\begin{theorem}\label{n}
The curvature $\Kl^0$ of the tautological connection $\lT^0$ satisfies the equation
$$2\kl^0(u,v)=-(d_ef)P_\kL[\t u,\t v]\qquad\forall\ u,v\in T_oM,$$ 
where $\t u\in\mL$ is uniquely determined by $(d_eR_o^M)\t u=u.$
\end{theorem}

\begin{proof} Let $\la,\lg\in\gL.$ Consider the left action vector field
$$(\ll_\la^Q)_{g|h}=\dl_t^0\,\<a_t g|h\>=\<(d_eR_g^G)\la|0_h\>$$
on $Q.$ The identity $R_g^G=L_g^G\circ I_{g^{-1}}^G$ implies $d_eR_g^G=(d_eL^G_g)Ad_{g^{-1}}^G.$ Therefore Proposition \ref{c} yields
$$(\lT^0\ll_\la^Q)_{g|e}=\lT^0_{g|e}\<(d_eR_g^G)\la|0_h\>=(d_ef)\,P_\kL\,Ad_{g^{-1}}^G\la.$$
Putting $g_t=\exp(t\,\lg)\in G$ we conclude that
$$(\lT^0\ll_\la^Q)_{g_t|e}=(d_ef)P_\kL\,Ad_{g_t^{-1}}^G\la=(d_ef)\,P_\kL\,
\exp(-t\,ad_\lg)\la$$ and hence
$$d_{e|e}(\lT^0\ll_\la^Q)\<\lg|0_e\>\,=\,\dl_t^0\,(\lT^0\ll_\la^Q)_{g_t|e}
$$
$$
=\,(d_ef)\,P_\kL\,\dl_t^0\exp(-t\,ad_\lg)\la\,=\,(d_ef)P_\kL[\la,\lg]\, .$$
Since $[\ll_\la^Q,\ll_\lg^Q]=\ll_{[\lg,\la]}^Q$ for left actions we have 
$$\lT^0_{e|e}[\ll_\la^Q,\ll_\lg^Q]_{e|e}=\lT^0_{e|e}(\ll^Q_{[\lg,\la]})_{e|e}=(d_ef)P_\kL\,[\lg,\la].$$ 
Therefore, we have
$$2\Kl^0_{e|e}(\<\la|0_e\>,\<\lg,0_e\>)=2(d\lT^0)_{e|e}(\<\la|0_e\>,\<\lg,0_e\>)+[\lT^0_{e|e}\<\la|0_e\>,\lT^0_{e|e}\<\lg|0_e\>]$$
$$=\,d_{e|e}(\lT^0\ll_\lg^Q)\<\la|0_e\>-d_{e|e}(\lT^0\ll_\la^Q)\<\lg|0_e\>
-\lT^0_{e|e}[\ll_\la^Q,\ll_\lg^Q]_{e|e}
$$
$$
+[\lT^0_{e|e}\<\la|0_e\>,\lT^0_{e|e}\<\lg|0_e\>]$$
$$=\,(d_ef)(P_\kL[\lg,\la])-(d_ef)(P_\kL[\la,\lg])-(d_ef)(P_\kL[\lg,\la])+
[(d_ef)(P_\kL\la),(d_ef)(P_\kL\lg)]
$$
$$
=\, (d_ef)([P_\kL\la,P_\kL\lg]- P_\kL[\la,\lg])\, .$$
Since $P_\kL\t u=0$ we obtain $2\kl^0(u,v)=2\lO^0_{e|e}(\<\t u|0\>,\<\t v|0\>)=-(d_ef)
(P_\kL[\t u,\t v]).$
\end{proof}

Every $\la\in\gL$ induces a left action vector field $\ll_\la^M$ on $M$ by putting
$$(\ll_\la^M)_x=(d_eR_x^M)\la=\dl_t^0(a_t x),$$
where $a_t=\exp(t\la).$ For left actions we have the reverse commutator identity $[\ll_\la^M,\ll_\lg^M]=\ll_{[\lg,\la]}^M$ for $\la,\lg\in\gL.$ 

\begin{lemma}\label{o}
For $\Lt\in\mL$ the vector field $\ll_\Lt^M$ has the horizontal lift 
\be{17}X^\Lt_{g|h}:=\<(d_eL_g^G)\,P_\mL\,Ad_{g^{-1}}^G\Lt|0_h\>.\ee
\end{lemma} 

\begin{proof} Let $(g,h)\in G\xx H$ and $k\in K.$ Then the equality
$R_{k^{-1}}^G\circ L_g^G=L_{gk^{-1}}^G\circ I_k^G$ implies that
$$(d_gR_{k^{-1}}^G)(d_eL_g^G)=(d_eL_{gk^{-1}}^G)Ad_k^G.$$ Since $[Ad_k^G,P_\mL]=0$, we have $Ad_k^G\,P_\mL\,Ad_{g^{-1}}^G=P_\mL\,Ad_k^G\,Ad_{g^{-1}}^G=P_\mL\,Ad_{kg^{-1}}^G$ and therefore
$$\<(d_gR_{k^{-1}}^G)(d_eL_g^G)\,P_\mL\,Ad_{g^{-1}}^G\Lt|(d_hL_{f(k)}^H)0_h\>
$$
$$
=\, \<(d_eL_{gk^{-1}}^G)\,Ad_k^G\,P_\mL\,Ad_{g^{-1}}^G\Lt|0_{f(k)h}\>
\,=\,\<(d_eL_{gk^{-1}}^G)\,P_\mL\,Ad_{kg^{-1}}^G\Lt|0_{f(k)h}\>\, .$$
This shows that \er{17} depends only on the class of $(g,h)$ and therefore defines a vector field on $Q$ which is $\lT^0$-horizontal by construction. Moreover
the equality $R_o^M\circ L_g^G=L_g^M\circ R_o^M$ implies that
$$(d_gR_o^M)(d_eL_g^G)=(d_oL_g^M)(d_eR_o^M),$$ and the equality
$L_g^M\circ R_o^M\circ I_{g^{-1}}^G=R_{go}^M$ implies that $(d_oL_g^M)(d_eR_o^M)\,Ad_{g^{-1}}^G=d_eR_{go}^M.$ Since 
$(d_eR_o^M)P_\mL=d_eR_o^M$ it follows that
$$(X^\Lt\lp)\<g|h\>=(d_{g|h}\lp)X^\Lt_{g|h}=(d_gR_o^M)(d_eL_g^G)\,P_\mL\,Ad_{g^{-1}}^G\Lt$$
$$=\, (d_oL_g^M)(d_eR_o^M)\,P_\mL\,Ad_{g^{-1}}^G\Lt\,=\,(d_oL_g^M)(d_eR_o^M)Ad_{g^{-1}}^G\Lt
$$
$$
=\, (d_eR_{go}^M)\Lt\,=\,(\ll_{\Lt}^M)_{go}\, .$$
\end{proof}

\section{The symmetric case}

Now we consider the special case where $M=G/K$ is a {\bf symmetric} space. These spaces have a well-known algebraic description using the so-called Lie triple systems \cite{H}. As discovered by M. Koecher \cite{Ko}, in the hermitian symmetric case there
is a more ``elementary'' approach using instead the so-called {\bf hermitian Jordan triple systems} \cite{L}. These are (complex) vector spaces $Z$ which carry a Jordan triple product 
$$(u,v,w)\mapsto\{uv^*w\}\qquad\forall\ u,v,w\in Z$$
which is symmetric bilinear in $(u,w)$ and conjugate-linear in $v.$ The Jacobi identity is replaced by the {\bf Jordan triple identity}
$$[u\Box v^*,z\Box w^*]=\{uv^*z\}\Box w^*-z\Box\{wu^*v\}^*\qquad\forall\ u,v,z,w\in Z.$$
Here $u\Box v^*\in\,End(Z)$ is defined by 
$$(u\Box v^*)z=\{uv^*z\}\qquad\forall\ z\in Z.$$
The basic example is the matrix space $Z={\mathbb C}^{r\xx s}$ with Jordan triple product
$$\{uv^*w\}=\f12(uv^*w+wv^*u).$$
The Jordan theoretic approach applies to all complex hermitian symmetric spaces, including the two exceptional types, and also to all classical real symmetric spaces. More generally, all real forms of complex hermitian symmetric spaces, for example symmetric convex cones \cite{FK}, and therefore also some exceptional real symmetric spaces, are included. (On the other hand, there exist non-classical real symmetric spaces which cannot be treated this way.)

We first consider both the real and complex case. Given a (real or complex) Jordan triple $Z,$ we put 
$$Q_zw:=\{zw^*z\}\qquad\forall\ z,w\in Z$$
and define the {\bf Bergman operator}
$$B_{z,w}=id_Z-2z\Box w^*+Q_zQ_w\qquad\forall\ z,w\in Z,$$
acting linearly on $Z.$ In case $B_{z,w}\in {\rm GL}(Z)$ is invertible, we define the {\bf quasi-inverse} 
$$z^w:=B_{z,w}^{-1}(z-Q_zw).$$
Define $\Le$ to be $\Le=-1$ for the non-compact case, $\Le=1$ for the compact case, and $\Le=0$ for the flat case. We define symmetric spaces $M^\Le$ associated with the Jordan triple $Z$ as follows: $M^0=Z$ is the flat model, $M^-$ is the connected component
$$M^-\ic\{z\in Z\,\mid\,\det B(z,z)\ne 0\}$$ 
containing the origin $o=0\in Z$ (a bounded symmetric domain, more precisely a norm unit ball of $Z$), and 
$$M^+=(Z\xx Z)/\os$$ 
is a compact manifold consisting of all equivalence classes $[z,a]=[z^{b-a},b],$ whenever $B(z,a-b)$ is invertible \cite{L}. (In view of the "addition formula" $(z^u)^v=z^{u+v}$ for quasi-inverses \cite{L}, we may informally regard $M^+$ as the set of all quasi-inverses $z^a,$ even when $B(z,a)$ is not invertible.) Thus we have natural inclusions
$$M^-\ic Z=M^0\ic M^+,$$
under the embedding $Z\ic M^+$ given by $z\mapsto z^0=[z,0].$ The points at infinity are precisely the classes $[z,a]$ where $\det B(z,a)=0.$ The compact dual $M^+$ is also called the {\bf conformal compactification} of $Z.$ At the origin the tangent space $T_oM=Z$ is independent of the choice of $\Le.$ Let $K\ic {\rm GL}(Z)$ be the identity component of the Jordan triple automorphism group of $Z,$ i.e., all linear isomorphisms of $Z$ preserving the Jordan triple product. The group $K$ acts by linear transformations on every type $M^\Le.$ For fixed $w\in Z\ui M^\Le,$ the (non-linear)
{\bf transvection}, defined by
\be{40}\tL_w^\Le(z):=w+B_{w,-\Le w}^{1/2}z^{\Le w},\ee
is a birational automorphism of $M^\Le.$ For $\Le=1,0,-1$ let $G^\Le$ denote the
connected real Lie group generated by $K$ together with the transvections \er{40}. In the
flat case $\Le=0$ we obtain the so-called Cartan motion group $G^0:=K\xx Z$ which is a
semi-direct product of $K$ and the translations $\tL_w^0(z)=z+w$ for $w\in Z.$ If $\Le\ne 0$,
then $G^\Le$ is a reductive Lie group of compact type ($\Le=1$) or non-compact
type ($\Le=-1$), respectively. We treat all three cases simultaneously, using the notation $M^\Le$ and $G^\Le$ to denote the curvature type. In all three cases we have
$$K=\{k\in G^\Le\,\mid\,ko=o\}.$$
In the Jordan theoretic setting, the Lie algebra $\gL^\Le$ of $G^\Le$ can be described using polynomial vector fields 
$$\lx=\lx(z)\f{\dl}{\dl z}$$
(or degree $\le 2$) on the underlying vector space $Z.$ More precisely, there is a Cartan decomposition
\be{23}\gL^\Le=\kL\oplus\pL^\Le,\ee 
where the Lie algebra $\kL\ic\mathfrak{gl}(Z)$ of $K,$ consisting of all Jordan triple derivations of $Z,$ is identified with a space of linear vector fields, whereas the Lie triple system $\pL^\Le$ consists of all non-linear vector fields
$$v^\Le:=v+\Le Q_zv=(v+\Le\{zv^*z\})\f{\dl}{\dl z}\qquad\forall\ v\in Z=T_oM^\Le.$$
The projection $\gL^\Le\to\kL$ is realized as the derivative 
$$P_\kL\lg=\lg'(o)$$
at the origin $o=0\in Z.$ One can show that $\exp v^\Le=\tL^\Le_w,$ where $w:=\tan_\Le v\in M^\Le$ is given by a ``tangent'' power series defined via the Jordan triple calculus \cite{L}.
\begin{proposition} In the symmetric case the tautological connection $\lT^0$ induced
by the Cartan decomposition \er{23} satisfies
\be{24}\kl^0(u,v)=-\Le(d_ef)(u\Box v^*-v\Box u^*)\qquad\forall\ u,v\in Z.\ee 
\end{proposition}
\begin{proof} Denoting by $u^\dl$ the (holomorphic) partial derivative in direction $u\in Z,$ we have the commutator identity
$$[u^\Le,v^\Le]=[u+\Le Q_zu,v+\Le Q_zv]=\Le\(u^\dl Q_zv-v^\dl Q_zu\)=2\Le(u\Box v^*-v\Box u^*)\in\kL.$$
Thus we obtain a linear vector field satisfying $[u^\Le,v^\Le]'(o)=[u^\Le,v^\Le],$ and Theorem \ref{n} implies $2\kl^0(u,v)=-df([u^\Le,v^\Le]'(o))=-2\Le\,df(u\Box v^*-v\Box u^*).$
\end{proof}

\begin{theorem}\label{q} In the symmetric case, let $\lT$ be an invariant connexion related to $\lT^0$ by an $f$-covariant linear map 
$\cl:T_oM\to\hL.$ Then the respective curvatures satisfy
$$\kl(u,v)-\kl^0(u,v)=\f12[\cl u,\cl v]\qquad\forall\ u,v\in T_oM.$$
\end{theorem}

\begin{proof} For $v\in Z$ let $X^v$ be the $\lT^0$-horizontal lift of $\ll_{v^\Le}^M,$ as in Lemma \ref{o}. Consider the tensorial $1$-form $\lF:=\lT-\lT^0$ on $Q.$ Then
$$\lO(X^u,X^v)-\lO^0(X^u,X^v)=d\lF(X^u,X^v)+\f12[\lF X^u,\lF X^v]\qquad\forall\ u,v\in T_oM.$$
Since $L_g^M\circ R_o^M\circ I_{g^{-1}}^G=R_{go}^M$ implies $(d_oL_g^M)(d_eR_o^M)Ad_{g^{-1}}^G=d_eR_{go}^M,$ it follows that
$$(\lF X^v)\<g|e\>\,=\,\Cl^Q_{g|e}(d_{g|e}\lp)X^v_{g|e}\,=\,\Cl^Q_{g|e}(d_eR_{go}^M)v^\Le
$$
$$
=\,\Cl^Q_{g|e}(d_oL_g^M)(d_eR_o^M)Ad_{g^{-1}}^Gv^\Le\,=\,\cl(d_eR_o^M)(Ad_{g^{-1}}^G)v^\Le\, ,$$
since $\Cl_{g|e}^Q(d_oL_g^M)=\cl$ as a special case of \er{16}. With $g_t:=\exp(t\,u^\Le)$ we obtain 
$$(\lF X^v)\<g_t|e\>=\cl(d_eR_o^M)(Ad_{g_{-t}}^G)v^\Le=\cl(d_eR_o^M)\exp(-t\,ad_{u^\Le})v^\Le.$$ 
Since the linear map $\cl(d_eR_o^M)$ commutes with taking the $\dl_t$-derivative, it follows that
$$X^u(\lF X^v)\<e|e\>=d_{e|e}(\lF X^v)\<u^\Le|0_e\>=\dl_t^0(\lF X^v)\<g_t|e\>$$
$$=\cl(d_eR_o^M)\dl_t^0\exp(-t\,ad_{u^\Le})v^\Le=-\cl(d_eR_o^M)\,ad_{u^\Le}v^\Le=-\cl(d_eR_o^M)\,[u^\Le,v^\Le]=0,$$
since $[u^\Le,v^\Le]\in[\pL^\Le,\pL^\Le]\ic\kL$ implies $(d_eR_o^M)[u^\Le,v^\Le]=0.$ On the other hand,
$$X^u(X^v\lp)\<e|e\>=d_{e|e}(X^v\lp)X^u_{e|e}=d_{e|e}(X^v\lp)\<u^\Le|0_e\>$$
$$=\dl_t^0(X^v\lp)\<g_t|e\>=\dl_t^0(\ll_{v^\Le}^M)_{g_to}=\dl_t^0(v+\Le Q_{g_t(0)}v)=0,$$
since applying the product rule to the quadratic term yields a term $g_0(0)=0.$. Thus
$$(d_{e|e}\lp)[X^u,X^v]_{e|e}=([X^u,X^v]\lp)\<e|e\>=\(X^u(X^v\lp)-X^v(X^u\lp)\)\<e|e\>=0$$
and hence
$$d\lF(X^u,X^v)\<e|e\>\,=\,X^u(\lF X^v)\<e|e\>-X^v(\lF X^u)\<e|e\>-\lF_{e|e}[X^u,X^v]_{e|e}
$$
$$=\, -\cl(d_{e|e}\lp)[X^u,X^v]_{e|e}\,=\, 0\, .$$
It follows that
$$\kl(u,v)-\kl^0(u,v)=\lO_{e|e}(X^u,X^v)-\lO_{e|e}^0(X^u,X^v)$$
$$=d\lF(X^u,X^v)\<e|e\>+\f12[(\lF X^u)\<e|e\>,(\lF X^v)\<e|e\>]=\f12[\cl u,\cl v].$$
\end{proof}
Now we consider the case where $M=G/K$ is an (irreducible) {\bf hermitian symmetric} space. In this case $Z$ is a complex hermitian
Jordan triple, $G^-$ is the identity component of the real Lie group of all biholomorphic automorphisms of $M^-,$ and $G^+$ is the identity component of the biholomorphic isometry group of $M^+,$ i.e., the biholomorphic automorphisms of $M^+$ that preserve the K\"ahler metric. Both Lie groups are semi-simple. The identity \er{24} can be polarized and hence {\bf defines} the Jordan triple product in terms of the curvature tensor at the base point $o.$

\begin{proposition}\label{s}
In the hermitian symmetric case, the tautological complexion $J^0$ is integrable.
\end{proposition}

\begin{proof} By Corollary \ref{p} we have to show that the curvature $2$-form $\Kl^0$ of the tautological connexion $\lT^0$ has vanishing $(0,2)$-part. By invariance under $G\xx H$ it suffices to consider the base point $\<e|e\>.$ The $\mathbb R$-bilinear map 
$$(u,v)\mapsto D(u,v)=u\Box v^*-v\Box u^*\in\kL,$$
for $u,v\in Z,$ satisfies $D(u,\sqrt{-1}v)+D(\sqrt{-1}u,v)=0$ and $D(\sqrt{-1}u,\sqrt{-1}v)=D(u,v).$ It follows that
$$D(u,v)+\sqrt{-1}D(\sqrt{-1}u,v)+\sqrt{-1}D(u,\sqrt{-1}v)-D(\sqrt{-1}u,\sqrt{-1}v)=0.$$
Therefore, Theorem \ref{q} shows that $$\kl^0(u,v)\,=\,\Le(d_ef)(u\Box v^*-v\Box u^*)
\,=\,\Le(d_ef)D(u,v)$$ has vanishing $(0,2)$-part.
\end{proof} 

If $\bl:Z\to\hL$ is an $f$-covariant anti-linear map, then $[\bl\yi\bl](u,v):=[\bl u,\bl v]$ defines an $f$-covariant anti-bilinear map
$$[\bl\yi\bl]:Z\yi Z\to\hL.$$

\begin{theorem} In the hermitian symmetric case, the invariant complexion $J$ generated by $\bl$ is integrable if and only if
$[\bl\yi\bl]=0.$
\end{theorem}

\begin{proof}
By Proposition \ref{l} the complexion $J$ is induced by the invariant connexion $\lT$
generated by $-\f{\sqrt{-1}}2\,\bl.$ Applying Theorem \ref{q} we have
$$\kl(u,v)-\kl^0(u,v)=\f12\,[\f{\sqrt{-1}}2\bl u,\f{\sqrt{-1}}2\bl v]=-\f18[\bl u,\bl v]$$
for all $u,v\in Z.$ By Theorem \ref{r} the Nijenhuis tensor for $J$ satisfies 
$$N_{e|e}(X,Y)=-2\o\kl((d_e\lp)X,(d_e\lp)Y)\qquad\forall\ X,Y\in T_{e|e}Q,$$ 
where $$\o\kl(u,v)=\kl(u,v)+\sqrt{-1}\kl(\sqrt{-1}u,v)+\sqrt{-1}\kl(u,\sqrt{-1}v)-
\kl(\sqrt{-1}u,\sqrt{-1}v)$$ is the $(0,2)$-part of the curvature $\kl$ of $\lT.$ Since the tautological complexion $J^0$ on $Q$ is integrable by Proposition \ref{s}, it follows that 
$$\o\kl^0(u,v)
$$
$$
=\, \kl^0(u,v)+\sqrt{-1}\kl^0(\sqrt{-1}u,v)+\sqrt{-1}\kl^0(u,\sqrt{-1}v)
-\kl^0(\sqrt{-1}u,\sqrt{-1}v)\,=\,0\, .$$ 
On the other hand, $[\bl\yi\bl]$ is of type $(0,2)$ since $\bl$ is anti-linear by Proposition \ref{k}. Therefore
$$\o\kl(u,v)=-\f18[\bl u,\bl v]$$
and hence 
$$N_{e|e}(X,Y)=\f14[\bl(d_e\lp)X,\bl(d_e\lp)Y]=\f14[\bl\yi\bl](X\xt Y)\qquad\forall\ X,Y\in T_{e|e}Q.$$ 
This implies that $N$ vanishes at the base point $\<e|e\>$ if and only if $[\bl\yi\bl]=0.$ By $G\xx H$-invariance, this implies that $N$ vanishes at all points $\<g|h\>\in Q.$
\end{proof}
\begin{theorem}\label{t}
In the hermitian symmetric case, the hermitian homogeneous $H$-bundles $Q$ endowed with an integrable invariant complexion $J$ are in $1-1$ correspondence with pairs $(f,\bl)$ consisting of a homomorphism $f:K\to L\ic H$ and an $f$-covariant 
anti-linear map $\bl:T_oM\to\hL$ satisfying
\be{26}[\bl\yi\bl]=0,\ee 
modulo the equivalence $(f,\bl)\os(I_\el^H\circ f,Ad_\el^H\,\bl),$ where $\el\in L$ is arbitrary. Via the relation \er{21}, we may equivalently consider pairs $(f,\al)$ where $\al:T_oM\to\Ll$ is any $f$-covariant $\mathbb R$-linear map such that
\be{27}[\al u,\al v]=[\al j_ou,\al j_ov]\qquad\forall\ u,v\in T_oM.\ee
\end{theorem}

\begin{proof} If $\al$ and $\bl$ are related by \er{21}, then the conditions \er{26} and \er{27} are equivalent.
\end{proof}

Since the classification of Theorem \ref{t} uses only data coming from the group $K,$ which is the same for the hermitian symmetric spaces
$M^\Le=G^\Le/K,$ for $\Le=0,1,-1,$ we obtain:

\begin{corollary}
There is a canonical $1-1$ correspondence between the sets classifying the hermitian 
homogeneous $H$-bundles $(P,Q)$, endowed with an integrable invariant complexion $J,$ 
over the hermitian symmetric spaces of compact type, non-compact type and flat type, 
respectively.
\end{corollary}

\section{Concluding remarks}

In order to put our results in perspective, and motivate the general treatment in the first 
four Sections, we outline possible generalizations of our approach. In the geometric 
quantization program, one considers more general non-symmetric $G$-homogeneous spaces 
$N=G/C,$ endowed with an invariant complex structure $j.$ An interesting class of examples 
is obtained as follows: Let $M=G/K$ be an irreducible hermitian symmetric space of compact, 
non-compact or flat type (depending on the choice of $G$). Then the tangent space 
$Z=T_o(M)$ at the origin $o\in M$ is a hermitian Jordan triple. The so-called principal 
inner ideals $U\ic Z$ are precisely the Peirce $2$-spaces $U=Z^2_c$ for a given tripotent 
$c\in Z.$ Let $F_j$ denote the Grassmann type manifold of all such Peirce $2$-spaces of 
fixed rank $j\le r.$ More generally, for any increasing sequence $1\le 
j_1<j_2<\ldots<j_\el\le r$ there is a flag type manifold $F_{j_1,\ldots,j_\el}$ consisting 
of all flags $U_1\ic U_2\ic\ldots\ic U_\el$ of Peirce $2$-spaces $U_i$ of rank $j_i.$ Using 
the $G$-action, one can define such flag manifolds for any tangent space $T_x(M),\,x\in M,$ 
and obtains a fibre bundle $N_{j_1,\ldots,j_\el}\to M$ with typical fibre 
$F_{j_1,\ldots,j_\el}.$ This fibre bundle is homogeneous 
$N_{j_1,\ldots,j_\el}=G/K_{j_1,\ldots,j_\el},$ where $K_{j_1,\ldots,j_\el}$ is a closed 
subgroup of $K$ which is the same for the compact, non-compact and flat case. Our principal 
result, concerning the classification of holomorphic principal fibre bundles and the 
duality between the compact and non-compact case, may be generalized to this setting 
\cite{BU}.

Another important problem, of interest in quantization theory, is the explicit construction of Dolbeault cohomology operators depending on the given complex structure, in the symmetric case or the more general setting outlined above. A first step towards this goal is a more explicit realization of the classifying space of holomorphic principal fibre bundles, described in Theorem \er{t} in the symmetric case. According to \er{26}, the basic case $H=GL_n(\Cl)$ involves pairwise commuting $n\xx n$-matrices $B_1,\ldots,B_d,$ where $d=\dim_\Cl G/K,$ modulo joint conjugation; already a quite complicated object in algebraic geometry. Finally, the whole construction depends on the underlying invariant complex structure $j$ of $N=G/C$ which may not be unique if $C$ is a proper subgroup of $K.$ Analogous to the Narasimhan-Seshadri Theorem for Riemann surfaces, the moduli space of invariant complex structures may carry a canonical projectively flat connexion on the bundle of holomorphic sections.

\end{document}